\documentclass[12pt,a4paper,reqno]{amsart}
\usepackage{amsmath}
\usepackage{amsthm}
\usepackage{amsfonts}
\usepackage{amssymb}
\usepackage{color}
\usepackage{mathrsfs}
\usepackage{graphicx}

\numberwithin{equation}{section}

\theoremstyle{plain}
\newtheorem{theorem}{Theorem}

\theoremstyle{plain}
\newtheorem*{theorema}{Theorem~A}

\theoremstyle{plain}

\theoremstyle{plain}
\newtheorem{lemma}{Lemma}[section]

\theoremstyle{remark}
\newtheorem*{remark}{Remark}

\theoremstyle{definition}
\newtheorem*{case1}{Case~1}

\theoremstyle{definition}
\newtheorem*{case2}{Case~2}

\theoremstyle{definition}
\newtheorem*{case11}{Case~1: colour split}

\theoremstyle{definition}
\newtheorem*{case22}{Case~2: no colour split}

\def\bfp{\mathbf{p}}

\def\bfv{\mathbf{v}}

\def\bfx{\mathbf{x}}

\def\Nn{\mathbb{N}}

\def\Qq{\mathbb{Q}}
\def\Rr{\mathbb{R}}

\def\Zz{\mathbb{Z}}

\def\EEE{\mathcal{E}}

\def\III{\mathcal{I}}

\def\LLL{\mathcal{L}}
\def\MMM{\mathcal{M}}

\def\PPP{\mathcal{P}}

\def\frakF{\mathfrak{F}}

\def\KKKK{\mathscr{K}}

\renewcommand{\le}{\leqslant}
\renewcommand{\ge}{\geqslant}

\title{A note on density of geodesics}

\author[Beck]{J. Beck}

\address{Department of Mathematics, Hill Center for the Mathematical Sciences, Rutgers University, Piscataway NJ 08854, USA}

\email{jbeck@math.rutgers.edu}

\author[Chen]{W.W.L. Chen}

\address{School of Mathematical and Physical Sciences, Faculty of Science and Engineering, Macquarie University, Sydney NSW 2109, Australia}

\email{william.chen@mq.edu.au}

\author[Yang]{Y. Yang}

\address{School of Science, Beijing University of Posts and Telecommunications, Beijing 100876, China}

\email{yangyx@bupt.edu.cn}

\begin{document}



\begin{abstract}
We extend the famous result of Katok and Zemlyakov on the density of half-infinite geodesics on finite flat rational surfaces
to half-infinite geodesics on a finite polycube translation $3$-manifold.
We also extend this original result to establish a weak uniformity statement.
\end{abstract}

\maketitle

\thispagestyle{empty}

%
%

\section{A density result}\label{sec1}

A finite \textit{polysquare region} $P$ is an arbitrary connected, though not necessarily simply-connected,
polygon on the plane which is tiled with closed unit squares, called the \textit{atomic squares} or \textit{faces} of~$P$,
and which satisfies the following conditions:

(i) Any two atomic squares in $P$ either are disjoint, or intersect at a single point, or have a common edge.

(ii) Any two atomic squares in $P$ are joined by a chain of atomic squares where any two neighbours
in the chain have a common edge.

Note that $P$ may have \textit{holes}, and we also allow \textit{whole barriers} which are vertical or
horizontal \textit{walls} that consist of one or more boundary edges of atomic squares, as shown
in the picture on the left in Figure~1.

\begin{displaymath}
\begin{array}{c}
\includegraphics[scale=0.8]{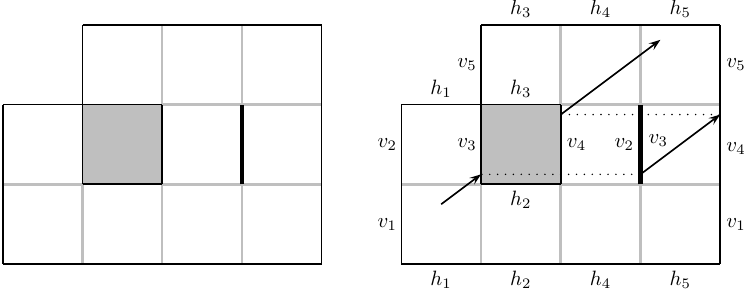}
\\
\mbox{Figure 1: bla bla bla}
\end{array}
\end{displaymath}

Given such a finite polysquare region~$P$, we can convert it into a finite \textit{polysquare translation surface} $\PPP$
by identifying pairs of the boundary edges with inward normals in opposite directions,
as illustrated in the passage from the picture on the left to the picture on the right in Figure~1.
Geodesic flow on such a finite polysquare translation surface $\PPP$ is then $1$-direction geodesic flow.

We say that a geodesic $\LLL(t)$, $t\ge0$, on $\PPP$ is \textit{half-infinite} if it does not hit a singularity of $\PPP$ and becomes undefined.
Furthermore, such a geodesic is \textit{dense} if it gets arbitrarily close to every point of~$\PPP$.

The following result of Katok and Zemlyakov~\cite{KZ75} in 1975
includes density of half-infinite geodesics on finite polysquare translation surfaces as a special case.

\begin{theorema}
With the exception of a countable set of given directions, every half-infinite geodesic on a finite rational surface $\PPP$ is dense.
In the special case when $\PPP$ is a finite polysquare translation surface, any half-infinite geodesic with irrational slope is dense in~$\PPP$.
\end{theorema}

We remark that a finite rational surface is a flat surface with finitely many faces and where every face is a rational polygon,
where each angle is a rational multiple of~$\pi$, 
and parallel edges with the same length are identified in pairs, while identified edges from the same face have
opposite orientation.

In this paper, we consider the $3$-dimensional analogue of the special case of this result.

A finite \textit{polycube region} $M$ is an arbitrary connected, though not necessarily simply-connected,
polyhedron in $3$-space which is tiled with closed unit cubes, called the \textit{atomic cubes} of~$M$,
and which satisfies the following conditions:

(i) Any two atomic cubes in $M$ either are disjoint, or intersect at a single point, or have a common edge, or have a common face.

(ii) Any two atomic cubes in $M$ are joined by a chain of atomic cubes where any two neighbours
in the chain have a common face.

Given such a finite polycube region~$M$, we can convert it into a finite \textit{polycube translation $3$-manifold} $\MMM$
by identifying pairs of the boundary faces with inward normals in opposite directions.
Geodesic flow in such a finite polycube translation $3$-manifold $\MMM$ is then $1$-direction geodesic flow.

A half-infinite geodesic in a finite polycube translation $3$-manifold$\MMM$ is dense if it gets arbitrarily close to every point of~$\MMM$.

To state our first result, we need a simple definition.

A vector $\bfv^*=(\alpha_1,\alpha_2,1)\in\Rr^3$ is a Kronecker direction if $v_1,v_2,1$ are linearly independent over~$\Qq$.

\begin{theorem}\label{thm1}
Let $\MMM$ be a polycube translation $3$-manifold with $s$ atomic cubes.
Then any half-infinite geodesic with a Kronecker direction $\bfv^*=(\alpha_1,\alpha_2,1)\in\Rr^3$ is dense in~$\MMM$.
\end{theorem}

\begin{proof}
Suppose, on the contrary, that $\LLL(t)$, $t\ge0$, is not dense.
Then its closure $\overline{\LLL}$ is a proper subset of~$\MMM$,
and the open set $\MMM\setminus\overline{\LLL}$ contains an open ball~$B$.
We may assume, without loss of generality, that $B$ is contained in a single atomic cube of~$\MMM$.
Then the subset
\begin{displaymath}
\III(B)=\{\bfv\in\MMM:\bfv=\bfv_0+t\bfv^*\mbox{ for some $\bfv_0\in B$ and $t\ge0$}\}\subset\MMM
\end{displaymath}
is a $\bfv^*$-flow-invariant open set.

Consider the multiplicity function $m:[0,1)^3\to\Nn$, given by
\begin{displaymath}
m(\bfx)=\vert\{\bfp\in\III(B):\bfp\equiv\bfx\bmod{1}\}\vert,
\quad
\bfx\in[0,1)^3.
\end{displaymath}
In other words, $m(\bfx)$ counts the number of points in $\III(B)$ that have common image $\bfx\in[0,1)^3$ under modulo~$1$ projection
from $\MMM$ to the unit torus $[0,1)^3$.
Since the $\bfv^*$-flow in the unit torus $[0,1)^3$ is ergodic, the function $m$ is constant almost everywhere,
and the constant is an integer $m_0$ satisfying $0\le m_0\le s$.

As $\III(B)$ is an open set, clearly $m_0\ge1$.
On the other hand, the complement of $\III(B)$ in $\MMM$ is a closed set that contains~$\overline{\LLL}$.
Clearly the image under modulo~$1$ projection of $\overline{\LLL}$ is $[0,1)^3$,
as the image of $\LLL(t)$, $t\ge0$, under modulo~$1$ projection from $\MMM$ to the unit torus $[0,1)^3$
is a geodesic in the unit torus $[0,1)^3$ with a Kronecker direction, and is therefore dense in $[0,1)^3$.
Hence $1\le m_0\le s-1$.

\begin{lemma}\label{lem11}
There exist $m_0$ open balls $B'_1,\ldots,B'_{m_0}$ of identical positive radius such that $B'_1\subset B$ and
\textcolor{white}{xxxxxxxxxxxxxxxxxxxxxxxxxxxxxx}
\begin{displaymath}
B'_i\subset\III(B),
\quad
i=1,\ldots,m_0,
\end{displaymath}
are in distinct atomic cubes of $\MMM$ and their images under modulo~$1$ projection from $\MMM$ to the unit torus $[0,1)^3$
are identical.
\end{lemma}

\begin{proof}
Let $B_0$ be the image of $B$ under modulo~$1$ projection from $\MMM$ to the unit torus $[0,1)^3$.
Clearly there exists $\bfx_0\in B_0$ such that $m(\bfx_0)=m_0$.
This means that there exist $\bfp_1\in B\subset\III(B)$ and $\bfp_2,\ldots,\bfp_{m_0}\in\III(B)$ in $m_0-1$
other distinct atomic cubes of $\MMM$ such that their common image 
under modulo~$1$ projection from $\MMM$ to the unit torus $[0,1)^3$ is~$\bfx_0$.

Since $\III(B)$ is an open set, it follows that for each $i=1,\ldots,m_0$, there exists a real number $r_i>0$
such that $B(\bfp_i;r_i)\subset\III(B)$, where $B(\bfp_i;r_i)$ denotes the open ball with centre $\bfp_i$ and radius~$r_i$.
Let $r=\min\{r_1,\ldots,r_{m_0}\}>0$.
Then the open balls
\begin{displaymath}
B'_i=B(\bfp_i;r)\subset\III(B),
\quad
i=1,\ldots,m_0,
\end{displaymath}
have common image $B(\bfx_0;r)\subset[0,1)^3$ under modulo~$1$ projection from $\MMM$ to the unit torus $[0,1)^3$.
\end{proof}

Let $B'_{m_0+1},\ldots,B'_s$ denote open balls in the other $s-m_0$ atomic cubes of $\MMM$
such that $B'_1,\ldots,B'_s$ have common image under modulo~$1$ projection from $\MMM$ to the unit torus $[0,1)^3$.
Observe that all these open balls have the same relative position within their own atomic cubes.

It is convenient to $2$-colour the polycube translation $3$-manifold $\MMM$ as follows.
We colour the $\bfv^*$-flow-invariant set $\III(B)$ \textit{white}
and colour its complement \textit{silver}.
Then each of the open balls $B'_1,\ldots,B'_{m_0}$ is almost everywhere white,
while each of the open balls $B'_{m_0+1},\ldots,B'_s$ is almost everywhere silver.

For technical reasons which will become clear later, we consider a collection
of open balls $B_1,\ldots,B_s$, where for every $i=1,\ldots,s$, the open ball $B_i$
has the same centre as the open ball~$B'_i$, but has radius equal to half of the radius of~$B'_i$.
Of course, each of the open balls $B_1,\ldots,B_{m_0}$ is almost everywhere white,
while each of the open balls $B_{m_0+1},\ldots,B_s$ is almost everywhere silver.

We next spread the open balls $B_1,\ldots,B_s$ with the $\bfv^*$-flow.
Suppose that $B_\dagger$ is the common image of $B_1,\ldots,B_s$ 
under modulo~$1$ projection from $\MMM$ to the unit torus $[0,1)^3$.
Then the projected $\bfv^*$-flow spreads $B_\dagger$ over the unit torus $[0,1)^3$ splitting free
such that every point in $[0,1)^3$ is covered.
This can be achieved in a sufficiently large finite time, in view of the Kronecker density theorem.

Consider now what happens to the open balls $B_1,\ldots,B_s$ under the $\bfv^*$-flow.
Since the system is in general non-integrable, there is splitting.
We make use of this splitting, and we first explain this using a very simple example.

Consider the polysquare translation surface with $4$ atomic squares in the picture on the left in Figure~2.
Here the vertical edges on the bottom row of atomic squares are barriers.
We now consider the polycube translation $3$-manifold $\MMM=\PPP\times[0,1)$
which is the cartesian product of the surface $\PPP$ with the unit torus $[0,1)$.
Here the $X$-faces on the bottom row of atomic cubes are barriers,
as shown in the picture on the right in Figure~2.

\begin{displaymath}
\begin{array}{c}
\includegraphics[scale=0.8]{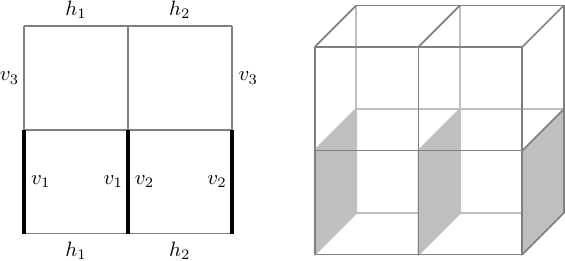}
\\
\mbox{Figure 2: a very simple polycube translation $3$-manifold $\MMM=\PPP\times[0,1)$}
\end{array}
\end{displaymath}

Figure~3, where the integrable direction corresponding to the unit torus $[0,1)$ is suppressed,
shows $4$ open balls $B_1,\ldots,B_4$ in the same relative positions within their own atomic cubes.
The lightgray colour shows part of their passages under the $\bfv^*$-flow
when they encounter singularities caused by the barriers on the bottom $X$-faces.
Let us see what happens to the open ball~$B_1$.
Clearly it is cut into $2$ parts by a plane containing the vector $\bfv^*$
and the top edge of barrier which is the middle bottom $X$-face.
The part labelled $B_1(+)$ continues unhindered and eventually reaches $B^*_1(+)$.
The part labelled $B_1(-)$ hits the middle bottom $X$-face, then bounces to the left bottom $X$-face
and eventually reaches $B^*_1(-)$.
However, $B^*_1(+)$ and $B^*_1(-)$ are now in different atomic cubes.
Instead, $B^*_1(+)$ now forms an open ball with $B^*_2(-)$.
In summary, each of $B_1,\ldots,B_4$ is cut into $2$ parts in an identical way,
and each individual part may after splitting pair up with the complementary part of a different open ball
to form an open ball.

\begin{displaymath}
\begin{array}{c}
\includegraphics[scale=0.8]{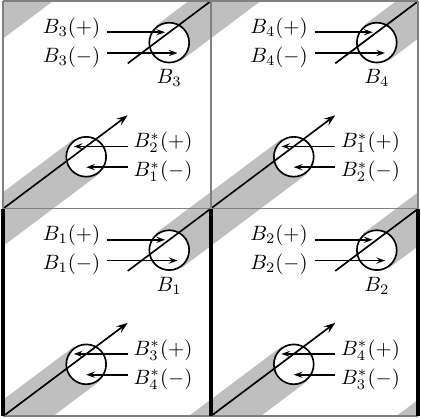}
\\
\mbox{Figure 3: explaining the splitting method on $\MMM=\PPP\times[0,1)$}
\end{array}
\end{displaymath}

After this example, we return to the general problem concerning the open balls $B_1,\ldots,B_s$.
Each of these open balls is monochromatic.
After encountering the splittings, we have $s$ new open balls, some of which may not be monochromatic any more,
as the $2$ individual parts of a new open ball may come from distinctly coloured open balls.

Recall that the projected $\bfv^*$-flow spreads $B_\dagger$ over the unit torus $[0,1)^3$ splitting free
such that every point in $[0,1)^3$ is covered, and this can be achieved in finite time.
Let us move the the open balls $B_1,\ldots,B_s$ by the $\bfv^*$-flow for this finite time,
and call this the finite $\bfv^*$-flow process.
Splitting will occur a finite number of times during this finite process.
There are two possibilities:

\begin{case1}
At some stage of the finite $\bfv^*$-flow process, there is a colour split open ball.
Thus a positive proportion of this ball is white, and a positive proportion of this ball is silver.
\end{case1}

\begin{case2}
After each splitting, all the new open balls remain monochromatic.
Thus there is no colour split balls at any stage of the finite $\bfv^*$-flow process.
\end{case2}

We show that both possibilities lead to a contradiction.

\begin{case11}
We first show that colour splitting is not reversible, in the sense that any further splitting cannot lead to a monochromatic open ball.

\begin{lemma}\label{lem12}
If there is a colour split open ball at any stage of the finite $\bfv^*$-flow process,
then there is a colour split open ball at the end of the finite $\bfv^*$-flow process.
\end{lemma}

\begin{proof}
It suffices to consider the effect of a splitting edge in the $y$-direction,
as the planes that cut $S$ into $2$ parts corresponding to splitting edges
in the $x$-, $y$- and $z$-directions are perpendicular to the directions
given by the vector products
\begin{displaymath}
\bfv^*\times(1,0,0),
\quad
\bfv^*\times(0,1,0),
\quad
\bfv^*\times(0,0,1),
\end{displaymath}
respectively, and so are not parallel.

Suppose that a point of an open ball $S$ hits some splitting edge
of the polycube translation $3$-manifold $\MMM$ in the $y$-direction
at a point $(m_1,y_1,n_1)$, where for the sake of convenience, we take $m_1,n_1\in\Zz$.
Then this point belongs to the plane that cuts $S$ into $2$ parts.
As $S$ continues in the direction $\bfv^*=(\alpha_1,\alpha_2,1)$ of the flow,
this point takes on values of the form
\begin{equation}\label{eq1.1}
(m_1,y_1,n_1)+t(\alpha_1,\alpha_2,1),
\quad
t\ge0.
\end{equation}
To establish the lemma, it suffices to show that the half-infinite geodesic \eqref{eq1.1}
never hits a splitting edge in the $y$-direction again.
Suppose, on the contrary, that it does.
Then there exist $t\in\Rr$, $m_2,n_2\in\Zz$ and $y_2\in\Rr$ such that
\begin{displaymath}
(m_1,y_1,n_1)+t(\alpha_1,\alpha_2,1)=(m_2,y_2,n_2),
\quad\mbox{so that}\quad
\alpha_1=\frac{m_2-m_1}{n_2-n_1},
\end{displaymath}
contradicting the assumption that $\bfv^*=(\alpha_1,\alpha_2,1)$ is a Kronecker direction.
\end{proof}

Let $B'_\dagger$ be the common image of $B'_1,\ldots,B'_s$ under modulo~$1$ projection
from $\MMM$ to the unit torus $[0,1)^3$.
In view of the Kronecker density theorem, the centre of the open ball $B_\dagger$ will
move under the projected $\bfv^*$-flow in the unit torus $[0,1)^3$ to a point close to
the centre of the open ball~$B'_\dagger$.
This means that the open ball $B_\dagger$ will move under the projected $\bfv^*$-flow in the unit torus $[0,1)^3$
to inside the bigger open ball~$B'_\dagger$.

The implication of this is that the colour split open ball will eventually move under the $\bfv^*$-flow in $\MMM$
to within one of the monochromatic open balls $B'_1,\ldots,B'_s$,
and this is clearly absurd.
\end{case11}

\begin{case22}
In this case, the $\bfv^*$-flow always exhibits $m_0$ white open balls and $s-m_0$ silver open balls.
The $\bfv^*$-flow spreads the open balls $B_1,\ldots,B_s$ and eventually covers the whole
polycube translation $3$-manifold~$\MMM$.
Within any atomic cube of~$\MMM$, consider any two overlapping open balls that are the $\bfv^*$-flow images
of two of $B_1,\ldots,B_s$.
Since both are monochromatic, they must have the same colour.
Thus this colour-coincidence spreads within the atomic cube,
and so it follows that every atomic cube of $\MMM$ is monochromatic.
Thus there are $m_0$ white atomic cubes and $s-m_0$ silver atomic cubes.
However, since $\MMM$ is face-connected, there must be $2$ neighbouring atomic cubes
where one is white and the other is silver.
Then the $\bfv^*$-flow transports a set of positive volume from an atomic cube of one colour
to an atomic cube of another colour, contradicting that the colour is $\bfv^*$-flow invariant.
\end{case22}

This completes the proof of the theorem.
\end{proof}

%
%

\section{When faces have gates and barriers}\label{sec2}

The idea of the proof of Theorem~\ref{thm1} can be extended to prove the following more general result.

Let $\MMM$ be modified from a finite polycube translation $3$-manifold as follows.
Each square face is in part a barrier, coloured red (shaded), and in part a gate, coloured green (white).
The latter permits travel between the two cubes that share the common square face.
We then modify boundary identification accordingly to ensure that $\MMM$ is boundary free
and remains connected.

\begin{theorem}\label{thm2}
Let $\MMM$ be any polycube $3$-manifold with $s$ atomic cubes and barriers,
where each square face has a $2$-colouring such that each of the red and green parts
is the union of finitely many polygons and $\MMM$ is connected.
Then for almost every direction $\bfv^*\in\Rr^3$, any half-infinite geodesic with direction $\bfv^*$
is dense in~$\MMM$.
\end{theorem}

\begin{proof}
We need to modify the proof of Theorem~\ref{thm1} appropriately at a number of points.
These concern those planes that cut the open balls $B$ into $2$ parts.
Apart from arising from the edges of $\MMM$, these now also arise from the edges of the polygons
on the faces of~$\MMM$.
Let us first deal with the latter, and let us restrict ourselves first to those that are on the $X$-faces of~$\MMM$.

The collection of polygons on one $X$-face may differ significantly from that on another $X$-face.
To handle this, we first observe that not all edges of the atomic cubes of $\MMM$ are singularities,
but we still consider \textit{splitting} on them.
If no splitting actually occurs, then each part of an open ball after splitting reunites with its complement.

Let $\Gamma_i$ denote the collection of colour split boundaries arising from the green and red polygons
on the right $X$-face of the $i$-th atomic cube of~$\MMM$,
and let $\Gamma'_i$ denote the image of $\Gamma_i$ under modulo~$1$ projection from $\MMM$ to
the unit torus $[0,1)^3$.
Then $\Gamma'_i$ lies on the right $X$-face of the unit torus $[0,1)^3$.
Let
\begin{displaymath}
\Gamma=\bigcup_{i=1}^s\Gamma'_i,
\end{displaymath}
and now replace $\Gamma_i$ on the right $X$-face of the $i$-th atomic cube by~$\Gamma^*_i$,
where the image of $\Gamma^*_i$ under modulo~$1$ projection from $\MMM$ to
the unit torus $[0,1)^3$ is~$\Gamma$.
Thus all the right $X$-faces of the atomic cubes of $\MMM$ have essentially the same collection
of \textit{splitting edges} arising from the green and red polygons.
We refer to these as the \textit{$X$-face splitting edges}.

Similarly there are \textit{$Y$-face splitting edges} and \textit{$Z$-face splitting edges}.

We have already commented at the beginning of the proof of Lemma~\ref{lem12}
that the planes that cut $S$ into $2$ parts corresponding to splitting edges
in the $x$-, $y$- and $z$-directions are not parallel.
Consider now an $X$-face splitting edge~$E$.
The plane corresponding $E$ that cuts $S$ into $2$ parts contains both $\bfv^*$ and~$E$.
We need to be careful in case this plane coincides with one of the following:

(1) the plane corresponding to an edge of an atomic cube in direction~$z$;

(2) the plane corresponding to an edge of an atomic cube in direction~$y$;

(3) the plane corresponding to an edge of an atomic cube in direction~$x$;

(4) the plane corresponding to a $Y$-face splitting edge $E_y$;

(5) the plane corresponding to a $Z$-face splitting edge $E_z$;

(6) the plane corresponding to a parallel $X$-face splitting edge $E_x$; or

(7) the plane corresponding to a non-parallel $X$-face splitting edge.

In case (1), the common plane must contain $\bfv^*$, $E$ and the edge of the atomic cube in direction~$z$.
This can only happen if $E$ is in the direction $z$ and the $\bfv^*$-flow takes $E$ to this edge.
It follows that the vector $\bfv^*$ can take only one of countably many different directions in~$\Rr^3$.
A similar conclusion can be drawn in case (2).

In case (3), the common plane must contain $\bfv^*$, $E$ and the edge of the atomic cube in direction~$x$.
Note that $E$ is not parallel to the edge of the atomic cube in direction~$x$.
This case can therefore only happen if the vector $\bfv^*$ lies on the plane
containing $E$ and the edge of the atomic cube in direction~$x$.
This gives rise to a set of directions in $\Rr^3$ of measure~$0$.

In case (4), the common plane must contain $\bfv^*$, $E$ and~$E_y$.
Note that $E$ is not parallel to~$E_y$.
This case can therefore only happen if the vector $\bfv^*$ lies on the plane
containing $E$ and~$E_y$.
This gives rise to a set of directions in $\Rr^3$ of measure~$0$.
A similar conclusion can be drawn in case (5).

In case (6), the common plane must contain $\bfv^*$, $E$ and~$E_x$.
Since $E$ and $E_x$ are parallel, this can only happen if the $\bfv^*$-flow takes $E$ to~$E_x$.
It follows that the vector $\bfv^*$ can take only one of countably many different directions in~$\Rr^3$.

In case (7), the plane containing the two $X$-face splitting edges has constant $x$-coordinate,
and so can only contain vectors of the form $(0,\alpha_1,\alpha_2)\in\Rr^3$.
These give rise to a set of directions in $\Rr^3$ of measure~$0$.

Note that there are only finitely many $X$-, $Y$- or $Z$-face splitting edges.
Hence the above give rise to a set of exceptional directions in $\Rr^3$ of measure~$0$.
In order to complete the proof of Theorem~\ref{thm2},
it suffices to establish the following analogue of Lemma~\ref{lem12}.

\begin{lemma}\label{lem21}
Let $E$ be a $Z$-face splitting edge, and
let $\EEE(E)$ denote the set of exceptional directions $\bfv^*=(\alpha_1,\alpha_2,1)\in\Rr^3$ such that
the $\bfv^*$-flow takes some point of $E$ back to $E$ or to some $Z$-face splitting edge analogous to~$E$.
Then the set $\EEE(E)$ has measure~$0$.
\end{lemma}

\begin{proof}
The edge $E$ can be described by a linear equation $c_1x+c_2y=c_3$,
where the coefficients $c_1,c_2,c_3\in\Rr$ are fixed.
Consider a point $(x_1,y_1,m_1)$ on~$E$,
where $x_1,y_1\in\Rr$ and for the sake of convenience, we take $m_1\in\Zz$.
Then
\begin{displaymath}
c_1x_1+c_2y_1=c_3.
\end{displaymath}
The effect of the $\bfv^*$-flow on this point is described by the half-infinite geodesic
\begin{displaymath}
(x_1,y_1,m_1)+t(\alpha_1,\alpha_2,1),
\quad
t\ge0.
\end{displaymath}
Suppose that this half-infinite geodesic takes the point $(x_1,y_1,m_1)$
back to $E$ or to some $Z$-face splitting edge analogous to~$E$.
Then there exist $t\in\Rr$, $m_2,n_2,q_2\in\Zz$ and $x_2,y_2\in\Rr$ such that
\begin{displaymath}
(x_1,y_1,m_1)+t(\alpha_1,\alpha_2,1)=(x_2,y_2,m_2)
\quad\mbox{and}\quad
c_1(x_2-n_2)+c_2(y_2-q_2)=c_3,
\end{displaymath}
from which we deduce that
\begin{equation}\label{eq2.1}
c_1(m_2-m_1)\alpha_1+c_2(m_2-m_1)\alpha_2=c_1n_2+c_2q_2,
\end{equation}
where $c_1$ and $c_2$ are fixed and $m_1$ is determined by~$E$.
For any fixed choice of the integers $(m_2,n_2,q_2)$, the equation \eqref{eq2.1}
shows that the exceptional points $(\alpha_1,\alpha_2)$ must lie on a line.
Since there are only countably many choices for the integers $(m_2,n_2,q_2)$,
it follows that the set $\EEE(E)$ can be described as a countable union of lines,
and so has measure~$0$.
\end{proof}

This completes the proof of Theorem~\ref{thm2}.
\end{proof}

\begin{remark}
We can compare Lemmas \ref{lem12} and~\ref{lem21}.
In the former, we see from the proof that the exceptional set of directions are the non-Kronecker directions,
whereas in the latter, this is not always the case.

Suppose that in the statement of Theorem~\ref{thm2}, we add an extra condition
and require the edges of the finitely many polygons to have \textit{rational} slopes.
Then in the proof of Lemma~\ref{lem21}, the coefficients $c_1$ and $c_2$
of the line $c_1x+c_2y=c_3$ can be taken to be rational.
Then the equation \eqref{eq2.1} becomes an equation in the variables $\alpha_1$ and $\alpha_2$
with rational coefficients, so that any solution leads to a non-Kronecker direction.

This observation suggests that face splitting edges with rational slopes behave in a similar fashion
to edges of the atomic cubes, whereas face splitting edges with irrational slopes are somewhat
more delicate.
\end{remark}

%
%

\section{Weak uniformity with bounded returns}\label{sec3}

For Theorem~A in its full generality, we can go beyond the class of finite flat rational surfaces,
and consider the larger class of finite \textit{translation surfaces}, where the restriction on the angles is dropped.
A finite translation surface $\PPP$ is a flat surface with finitely many faces and where every face is a polygon,
and parallel edges with the same length are identified in pairs, while identified edges from the same face have
opposite orientation.

In this generality, the directions in $\PPP$ such that no geodesic on $\PPP$ can contain more than one vertex of $\PPP$
defines the \textit{good} directions
and plays the perfect analogue to irrational slopes in the special case when $\PPP$ is a polysquare translation surface.
It remains to show that the set of \textit{bad} directions is countable.
A geodesic on the unit torus $[0,1)^2$ can be extended to a straight line on the plane.
Here every geodesic on $\PPP$ containing at least two vertices of $\PPP$ can be extended to a straight line
on a countable covering of the plane.
Each plane on this countable covering contains a countable number of extended vertices.
Since two points determine a straight line, the set of bad directions is countable.

For a finite translation surface~$\PPP$, we can identify the bad directions in terms of \textit{saddle connections}.
A finite geodesic segment on a finite translation surface $\PPP$ is called a saddle connection if both endpoints
of the geodesic segment are vertices of $\PPP$ and there are no vertices of $\PPP$ in between.
The set of saddle connections of a finite translation surface is countable.
Furthermore, the collection of saddle connections with a given fixed slope is finite.
Then the exceptional directions in Theorem~A are precisely those directions of the saddle connections of~$\PPP$.

Let $\PPP$ be a finite translation surface.
Density of a half-infinite geodesic on~$\PPP$, as given by the full generality of Theorem~A, means that
the orbit of the geodesic visits every non-empty open set in~$\PPP$.
Since the orbit is infinite, this trivially means that it returns to this open set infinitely many times.

A less trivial consequence of density of a half-infinite geodesic is a new concept
which we call \textit{weak uniformity with bounded returns}.
This means that for every non-empty open set $G\subset\PPP$,
there is a constant $\frakF=\frakF(G;\PPP;\alpha)>0$, where $\alpha$ is the slope of the geodesic,
and a finite threshold $c_5=c_5(G;\PPP;\alpha)>0$ such that
for every geodesic segment $L$ with slope $\alpha$ and length $\vert L\vert\ge c_5$,
the inequality
\begin{displaymath}
\lambda_1(G\cap L)\ge\frakF\vert L\vert
\end{displaymath}
holds, where $\lambda_1$ denotes $1$-dimensional Lebesgue measure,
so that $L$ visits $G$ in total length at least $\frakF\vert L\vert$.
Intuitively, $\frakF>0$ gives a lower bound to what may be called the \textit{visiting frequency} of $G$
relative to geodesic segments of slope~$\alpha$.

Every open set in a finite rational translation surface $\PPP$ is Lebesgue measurable.
However, some open sets $G\subset\PPP$ are not Jordan measurable,
and the characteristic function $\chi_G$ has no well defined $2$-dimensional Riemann integral.
The traditional Weyl type uniformity requires that the asymptotic visiting frequency of $G$
is precisely equal to the relative volume of $G$, assuming that $G$ is Jordan measurable.
This motivates us to describe this new concept as a weak form of uniformity.

We establish the following result.

\begin{theorem}\label{thm3}
Suppose that $\alpha$ is not the slope of any saddle connection of a finite translation surface~$\PPP$.
Then every non-empty open set $G\subset\PPP$ has positive visiting frequency relative to geodesic segments of slope~$\alpha$.
\end{theorem}

Before we start the proof, let us draw some simple and useful consequences from Theorem~A.

Let $\alpha$ be fixed and not equal to the slope of any saddle connection of~$\PPP$.
Consider the direction $(1,\alpha)$.
A point $Q\in\PPP$ is said to be a pathological starting point in this direction if a geodesic in this direction
and starting from $Q$ hits a singular vertex of $\PPP$ and becomes undefined.
Otherwise $Q\in\PPP$ is said to be a non-pathological starting point in this direction.

Let $Q\in\PPP$ be a non-pathological starting point of a half-infinite geodesic $\LLL(t)$, $t\ge0$,
in the direction $(1,\alpha)$, so that $\LLL(0)=Q$.
Then Theorem~A implies that for every non-empty open set $G\subset\PPP$,
there is a finite threshold $T=T(G;Q)$ such that
\begin{displaymath}
\{\LLL(t):0\le t\le T(G;Q)\}\cap G\ne\emptyset,
\end{displaymath}
so that a geodesic segment of length $T=T(G;Q)$ in the direction $(1,\alpha)$ and starting from the point $Q$
visits the set~$G$.

\begin{proof}[Proof of Theorem~\ref{thm3}]
Suppose that $\alpha$ is fixed and not equal to the slope of any saddle connection of~$\PPP$.
Let $G_1\subset\PPP$ be an open disk of radius $r>0$, and consider the set
\begin{equation}\label{eq3.1}
\{T(G_1;Q):Q\in\PPP\mbox{ not pathological in direction }(1,\alpha)\}.
\end{equation}
We first show that this set is bounded.

Suppose, on the contrary, that the set \eqref{eq3.1} is not bounded.
Then there exists a sequence of points $Q_i\in\PPP$, $i=1,2,3,\ldots,$ such that $T(G_1;Q_i)\to\infty$ as $i\to\infty$.
Since $\PPP$ is compact, there exists a subsequence $Q_{i_j}$, $j=1,2,3,\ldots,$ such that 
$Q_{i_j}\to Q_\infty$ as $j\to\infty$ for some point $Q_\infty\in\PPP$ and
\begin{equation}\label{eq3.2}
T(G_1;Q_{i_j})\to\infty
\quad
\mbox{as $j\to\infty$}.
\end{equation}
The limit point $Q_\infty\in\PPP$ may be non-pathological or pathological.

Suppose that $Q_\infty\in\PPP$ is pathological.
This means that either $Q_\infty$ is a singular vertex of $\PPP$ and so there is no unique geodesic in the direction $(1,\alpha)$
that starts from~$Q_\infty$, or the geodesic in the direction $(1,\alpha)$ and starting from $Q_\infty$ hits a singular vertex of $\PPP$
and becomes undefined thereafter.
However, there are only finitely many ways to extend the geodesic in the direction $(1,\alpha)$ and starting from $Q_\infty$ beyond the singular vertex.

Let $G_2\subset\PPP$ be an open disk of radius $r/2$ and with the same centre as~$G_1$.

Suppose that $\LLL^{(\kappa)}$ is one such extension of the geodesic from~$Q_\infty$.
Let $Q^{(\kappa)}\in\PPP$ be a point beyond the singular vertex on~$\LLL^{(\kappa)}$.
Since $\alpha$ is not the slope of any saddle connection of~$\PPP$, it follows that $Q^{(\kappa)}$ is non-pathological.
Let $d=d(Q_\infty,Q^{(\kappa)})$ be the length of the extended geodesic segment of $\LLL^{(\kappa)}$ from $Q_\infty$ to~$Q^{(\kappa)}$.
Then an initial geodesic segment of $\LLL^{(\kappa)}$ of length $d+T(G_2;Q^{(\kappa)})$ visits~$G_2$.
Furthermore, apart from very short segments near the singular vertex it intersects, this finite geodesic segment of $\LLL^{(\kappa)}$
from $Q_\infty$ to $G_2$ has a positive distance $\delta^{(\kappa)}$ from other singular vertices of~$\PPP$.
Let $\delta_\kappa=\min\{\delta^{(\kappa)},r/2\}$.

Consider next a half-infinite geodesic $\LLL_j$ in the direction $(1,\alpha)$ and starting from a point $Q_{i_j}$ in the
$\delta_\kappa$-neighbourhood of~$Q_\infty$.
If $\LLL_j$ does not split from $\LLL^{(\kappa)}$ up to~$Q^{(\kappa)}$, then it does not split from $\LLL^{(\kappa)}$ before it reaches~$G_1$.
More precisely, if an extended geodesic segment of length $d+T(G_2;Q^{(\kappa)})$ starting from $Q_\infty$ reaches a point $Q_\infty^*\in G_2$,
then a geodesic segment of length $d+T(G_2;Q^{(\kappa)})$ starting from $Q_{i_j}$ reaches a point in the $\delta_\kappa$-neighbourhood of $Q_\infty^*$.
It is clear that any point in this $\delta_\kappa$-neighbourhood of $Q_\infty^*$ is a distance less that $\delta_\kappa+r/2\le r$ from the centre of~$G_2$,
which is also the centre of~$G_1$, and so is contained in~$G_1$.
It follows that
\begin{equation}\label{eq3.3}
T(G_1;Q_{i_j})\le d+T(G_2;Q^{(\kappa)}).
\end{equation}

There are only finitely many possible extensions $\LLL^{(\kappa)}$, $\kappa\in\KKKK$, of the geodesic from~$Q_\infty$.
Let
\textcolor{white}{xxxxxxxxxxxxxxxxxxxxxxxxxxxxxx}
\begin{displaymath}
\delta^*=\min_{\kappa\in\KKKK}\delta_\kappa>0.
\end{displaymath}
Then there are clearly infinitely many values of $j=1,2,3,\ldots$ such that $Q_{i_j}$ is in the $\delta^*$-neighbourhood of~$Q_\infty$.
It follows that there exists $\kappa\in\KKKK$ such that the inequality \eqref{eq3.3} holds for infinitely many values of~$j$.
This clearly contradicts \eqref{eq3.2}.

For the case when $Q_\infty$ is not pathological, there is a half-infinite geodesic in the direction $(1,\alpha)$ that starts from~$Q_\infty$.
This essentially means that there is only one value of~$\kappa$, and we simply take $Q^{(\kappa)}=Q_\infty$ and $d=0$.

This completes the deduction that the set \eqref{eq3.1} is bounded.
It is also clear that a similar argument shows that both sets
\begin{displaymath}
\{T(G_2;Q):Q\in\PPP\mbox{ not pathological in direction }(1,\alpha)\}
\end{displaymath}
and
\textcolor{white}{xxxxxxxxxxxxxxxxxxxxxxxxxxxxxx}
\begin{displaymath}
\{T(G_2;Q):Q\in\PPP\mbox{ not pathological in direction }(-1,-\alpha)\}
\end{displaymath}
are bounded.
Let $T^*(G_2)$ be a finite upper bound of the two sets.

A point $Q\in\PPP$ cannot be pathological in both directions $(1,\alpha)$ and $(-1,-\alpha)$, unless $Q$ is a singular vertex of~$\PPP$.
It then follows that every geodesic segment of slope $\alpha$ and length $T^*(G_2)$ that does not contain a singular vertex of $\PPP$
visits~$G_2$, and so visits $G_1$ in total length at least~$r/2$.

Next, let $G\subset\PPP$ be a non-empty open set.
Then there exists an open disk $G_1\subset G$ with some radius $r>0$.
Let $L$ be a geodesic segment on $\PPP$ with length
\begin{equation}\label{eq3.4}
\vert L\vert\ge2T^*(G_2).
\end{equation}
If $L$ contains a singular vertex of~$\PPP$, then this can only be one of the two endpoints.
It follows that $L$ contains at least
\begin{displaymath}
\left[\frac{\vert L\vert}{T^*(G_2)}\right]-1=k\ge1
\end{displaymath}
distinct geodesic segments $L_1,\ldots,L_k$, each of length $T^*(G_2)$ and not containing a singular vertex of~$\PPP$.
Each of $L_1,\ldots,L_k$ visits $G_1$ in total length at least~$r/2$, so $L$ visits $G$ in total length at least
\begin{equation}\label{eq3.5}
\frac{rk}{2}
=\frac{r}{2}\left(\left[\frac{\vert L\vert}{T^*(G_2)}\right]-1\right)
\ge\frac{r}{4}\left[\frac{\vert L\vert}{T^*(G_2)}\right]
\ge\frakF\vert L\vert,
\end{equation}
where
\textcolor{white}{xxxxxxxxxxxxxxxxxxxxxxxxxxxxxx}
\begin{equation}\label{eq3.6}
\frakF=\frakF(G;\PPP;\alpha)=\frac{r}{8T^*(G_2)}.
\end{equation}
The theorem follows on combining \eqref{eq3.4}--\eqref{eq3.6}.
\end{proof}

Using similar ideas, we can establish the following result.

\begin{theorem}\label{thm4}
Let $\MMM$ be a finite polycube translation $3$-manifold, and let $\bfv$ be a Kronecker direction.
Then every non-empty open set $G\subset\MMM$ has positive visiting frequency.
\end{theorem}

%
%


\begin{thebibliography}{9}

\bibitem{KZ75}
A. Katok, A. Zemlyakov.
Topological transitivity of billiards in polygons.
\textit{Math. Notes} \textbf{18} (1975), 760--764.

\end{thebibliography}
\end{document}